\documentclass{ifacconf}

\usepackage{graphicx}      
\usepackage{natbib}        
\usepackage{amsmath}
\usepackage{amssymb}
\let\theoremstyle\relax
\usepackage{amsthm}
\usepackage{booktabs}
\usepackage{balance}

\begingroup
    \makeatletter
    \@for\theoremstyle:=definition,remark,plain\do{%
        \expandafter\g@addto@macro\csname th@\theoremstyle\endcsname{%
            \addtolength\thm@preskip\parskip
            }%
        }
\endgroup

\theoremstyle{plain}
\newtheorem{lemma}{Lemma}
\newtheorem{theorem}{Theorem}

\newtheorem{remark}{Remark}
\newtheorem{assumption}{Assumption}

\theoremstyle{definition}

\usepackage{xcolor}

\begin{document}
\begin{frontmatter}

\title{
Direct Data-Driven Approximate Optimal Control of Nonlinear Input-Affine Systems\thanksref{footnoteinfo}} 

\thanks[footnoteinfo]{The work of B. Nortmann is supported as part of NCCR Automation, a National Centre of Competence in Research, funded by the Swiss National Science Foundation (grant number 51NF40\_225155).}

\author[Empa]{B. Nortmann} 
\author[Imperial]{T. Mylvaganam} 

\address[Empa]{Urban Energy Systems Laboratory, Swiss Federal Laboratories for Materials Science and Technology, Dübendorf, Switzerland (e-mail: benita.nortmann@empa.ch).}
\address[Imperial]{Department of Aeronautics, Imperial College London, London, UK (e-mail: tmylvaganam@imperial.ac.uk)}

\begin{abstract}                
In this paper, we combine a data-driven system representation with a framework to systematically construct (approximate) solutions to nonlinear optimal control problems.
By immersing the unknown dynamics into an extended state space, solutions are characterised via purely data-dependent algebraic conditions.
This allows us to design dynamic state-feedback controllers with local stability and performance guarantees for unknown nonlinear, input-affine systems directly using data, without explicitly identifying the dynamics.
\end{abstract}

\begin{keyword}
Data-based control, optimal control, stability of nonlinear systems
\end{keyword}

\end{frontmatter}

\section{Introduction} \label{sec:intro}
The problem of designing controllers directly from measured data has been receiving increasing attention in recent years. 
A framework to design state-feedback controllers for linear time-invariant (LTI) systems, which utilises Willems' Fundamental Lemma (see \cite{Willems}) to represent the closed-loop system directly using data has been introduced in \cite{DePersis2020}.
Various approaches have been presented to extend the ideas to classes of nonlinear systems. For more details and the relation to nonlinear system identification, see \cite{DePersis2023review,Monshizadeh2025} and references therein. The main focus has so far been on data-driven stabilisation of nonlinear systems. Optimal control has been considered for bilinear systems (\cite{Yuan2022}) and for discrete input-affine systems via online linearisation (\cite{Dai2025}). Data-based (approximate) solutions to optimal control problems 
using different approaches to parametrise unknown nonlinear dynamics have been proposed in \cite{Yu2022}, \cite{Prag2022} and references therein. 

Even in the model-based case, solving infinite-horizon nonlinear optimal control problems is generally challenging. 
The solution is characterised via the Hamilton-Jacobi-Bellmann (HJB) equation (\cite{liberzon2011}).
With the exception of certain special cases, closed-form solutions of this nonlinear partial differential equation (PDE) are not available and numerical solutions are computationally expensive to obtain. 
For optimal control problems involving nonlinear, input-affine systems, \cite{Sassano2018} presents a framework to construct (approximate) solutions without involving the explicit solution of any PDE.
Namely, via an immersion of the nonlinear dynamics into an extended state-space, approximate solutions are constructed systematically, based solely on algebraic matrix conditions, with any matrix satisfying these conditions referred to as an \emph{algebraic $P$ solution.}   
Moreover, the level of approximation is quantifiable and, thus, also minimisable. 
More details about the framework and its relation to alternative methods for nonlinear optimal control can be found in \cite{Sassano2021} and references therein.

In this work, we combine the discussed line of direct data-driven control with the aforementioned ``framework of algebraic $P$ solutions".
This allows us to determine (approximate) solutions to nonlinear optimal control problems involving unknown input-affine systems directly using measured data. 
The algebraic $P$ solutions framework has been extended to other 
control problems involving PDEs, including classes of stochastic control problems (\cite{GONG2025}, $\mathcal{L}_2$-disturbance attenuation (\cite{Sassano2012}), and differential games (\cite{Mylvaganam2015}).
Hence, the presented approach lays the foundations for the direct data-driven solution of a wide class of nonlinear control problems. 

The paper is organised as follows.  
The considered problems are defined in Section~\ref{sec:problem}. A
data-driven system representation is introduced in Section~\ref{sec:data}. Section~\ref{sec:solution} 
proposes control design methods via purely data-dependent algebraic conditions. 
The results are demonstrated via an example in Section~\ref{sec:example}. Concluding remarks are provided in Section~\ref{sec:conclusion}.

\textbf{Notation.} The sets of real numbers and nonnegative real numbers are denoted by $\mathbb{R}$, and $\mathbb{R}_{\geq 0}$. 
Given a mapping $V: \mathbb{R}^n \rightarrow \mathbb{R}$ and a vector $x \in \mathbb{R}$, $V_x$ denotes the column vector containing the partial derivatives of $V$ with respect to each element of $x$, and $\mathbf{J}_{V}(x)$ denotes the Jacobian matrix of $V$ with respect to $x$. A mapping $f: \mathbb{R}^n \rightarrow \mathbb{R}^m$ is called smooth if it has derivatives of all orders, and $f$ is called analytic if it is smooth and the Taylor series expansion around every point $x_0$ equals $f(x)$ in a neighbourhood of $x_0$. Given a matrix $A$, its transpose is denoted by $A^\top$, and its Moore–Penrose inverse by $A^\dagger$. Given a square matrix $B$, $B \succ 0$ ($B \succeq 0$) denotes that $B$ is positive definite (positive semidefinite) and $B^{-1}$ denotes its inverse. The $n \times n$ identity matrix is denoted by $I_n$.

\section{Problem formulation} \label{sec:problem}
Consider the input-affine dynamical system 
\begin{equation}
    \dot x 
    = f(x) + g(x) u,
    \label{eq:sysdyn}
\end{equation}
where  $x \in \mathbb{R}^n$ is the state and $u \in \mathbb{R}^m$ is the control input, 
and let $x(0) = x_0$.

\begin{assumption}
\label{as:properties_dynamics}
The mapping $f: \mathbb{R}^{n} \rightarrow \mathbb{R}^n$ is analytic and  such that 
$ 
f(0) =0\,,
$
and the mapping $g: \mathbb{R}^n \rightarrow\mathbb{R}^{n \times m}$ is smooth.
\end{assumption}
\noindent  A consequence of Assumption \ref{as:properties_dynamics} is that there exist matrix-valued mappings $F: \mathbb{R}^n \rightarrow \mathbb{R}^{n\times n}$ such that $f(x) = F(x) x $.

We focus on designing state-feedback control strategies for the action $u$. The solution to various nonlinear feedback control problems, such as stabilisation and optimal control, 
involve the solution of PDEs, 
see \emph{e.g.} \cite{slotine1991,liberzon2011}.
In what follows we build on the results in \cite{Sassano2018}, which utilise algebraic conditions and system immersion -- resulting in dynamic feedback controllers -- to circumvent the need to solve PDEs. Hence, we consider 
(dynamic) 
control laws of the form
\begin{subequations}
\label{eq:dynamic-state-feedback}
    \begin{align}
        u &= \tilde h(x, \xi), 
        \label{eq:dynamic-state-feedback-u} \\
        \dot \xi &= \alpha (x,\xi), \label{eq:dynamic-state-feedback-xi}
    \end{align}
\end{subequations}
where $\xi \in \mathbb{R}^{\nu}$, $\tilde h(0,0) = 0$, $\alpha(0,0) = 0$, and $\tilde h: \mathbb{R}^n \times \mathbb{R}^{\nu} \rightarrow \mathbb{R}^m$ and $\alpha: \mathbb{R}^n \times \mathbb{R}^{\nu} \rightarrow \mathbb{R}^{\nu}$ are smooth mappings. The variable $\xi$ is referred to as the \emph{dynamic extension}, and its initial condition $\xi(0)=\xi_0$ can be seen as a design parameter. We consider the following control problems. 

\subsection{Stabilisation} \label{subsec:probstab}
Consider system \eqref{eq:sysdyn} and the problem of designing a feedback law of the form \eqref{eq:dynamic-state-feedback}, which stabilises the equilibrium at the origin. 
More precisely, this involves designing functions $\tilde h$ and $\alpha$, such that the zero equilibrium of the resulting extended closed-loop system
\begin{equation}
    \begin{bmatrix}
        \dot x \\
        \dot \xi
    \end{bmatrix} = \begin{bmatrix}
        f(x) + g(x) \tilde h(x, \xi) \\
        \alpha(x, \xi)
    \end{bmatrix}
    \label{eq:augmented_sys_dynamic_cl}
\end{equation}
is locally asymptotically stable (LAS).
Via standard Lyapunov arguments, \eqref{eq:augmented_sys_dynamic_cl} is LAS if there exists a smooth function $V: \mathbb{R}^n \times \mathbb{R}^\nu \rightarrow \mathbb{R}$ and a neighbourhood $\mathcal{N}$ containing the origin of $(x, \xi)$, such that $V(0,0)=0$ and
\begin{align}
\begin{cases}
    V(x,\xi) > 0,  \\
    \dot V(x,\xi) = \begin{bmatrix} V_x^\top V_{\xi}^\top \end{bmatrix} \begin{bmatrix}
        f(x) + g(x) \tilde h(x, \xi) \\
        \alpha(x, \xi)
    \end{bmatrix} < 0,
\end{cases} 
\label{eq:lyap_las}
\end{align}
for all  $ (x, \xi) \in \mathcal{N}\setminus \{0\}$. 
This gives the PDE condition
\begin{equation}
    V_x^\top (f(x)+g(x) \tilde h(x,\xi)) + V_{\xi}^\top \alpha(x,\xi) + \tilde q(x,\xi) =  0 ,
    \label{eq:stabilityPDE}
\end{equation}
where $\tilde q: \mathbb{R}^n \times \mathbb{R}^{\nu} \rightarrow \mathbb{R}_{\geq 0}$ is a positive definite function\footnote{That is, $\tilde q(0,0) = 0$ and $\tilde q(x,\xi) > 0$, $\forall x \neq 0 \in \mathbb{R}^n$ and $\forall \xi \neq 0 \in \mathbb{R}^{\nu}$.}.

\subsection{Approximate optimal control} \label{subsec:approximate_optimal}
Consider system \eqref{eq:sysdyn} and the cost functional
\begin{equation}
    J(x_0,u) = \frac{1}{2} \int_0^\infty \left( q(x) + u^\top u \right) dt,
    \label{eq:cost}
\end{equation}
with $q: \mathbb{R}^n \rightarrow \mathbb{R}_{\geq 0}$ a smooth mapping.

\begin{assumption}
    \label{as:cost}
    The function $q$ is positive-definite,
    and there exists a function $C: \mathbb{R}^n \rightarrow \mathbb{R}^{\mu \times n}$ such that $q(x) = x^\top C(x)^\top C(x) x$, for all $x$.
\end{assumption}

The ``classical'' infinite-horizon optimal control problem consists in designing a stabilising static feedback law $u = h(x)$, which minimises the cost \eqref{eq:cost}, subject to the dynamics \eqref{eq:sysdyn}. A solution can be found by solving the HJB PDE, see \emph{e.g.} \cite{liberzon2011}. With the aim of instead designing a dynamic feedback law towards optimising \eqref{eq:cost}, we consider the following auxiliary problem (which allows for potentially dynamic feedback control strategies), see also \cite{Sassano2021}.

Let $\mathcal{B} \subset \mathbb{R}^n \times \mathbb{R}^{\nu}$ be an open set containing the origin. Consider the augmented cost functional
\begin{equation}
    \tilde J(x_0,\xi_0,u) = \frac{1}{2} \int_0^\infty \left( q(x) + \Upsilon(x,\xi) + u^\top u \right) dt.
    \label{eq:augmented_cost}
\end{equation} 
where $\Upsilon: \mathbb{R}^n \times \mathbb{R}^\nu \rightarrow \mathbb{R}$ 
is such that $\Upsilon(0,0) = 0$ and $\Upsilon(x,\xi) \ge 0$, for all $(x, \xi) \in \mathcal{B}$. 
We then consider the (auxiliary) problem of designing a feedback law of the form \eqref{eq:dynamic-state-feedback}, such that the zero equilibrium of the resulting closed-loop system \eqref{eq:augmented_sys_dynamic_cl} is LAS with region of attraction containing $\mathcal{B}$, and such that 
\begin{equation}
    \tilde J(x_0,\xi_0,u^\star) \leq \tilde J(x_0,\xi_0,\tilde u), 
    \label{eq:optimality_condition}
\end{equation}
for all $\tilde u(x,\xi)$ of the form \eqref{eq:dynamic-state-feedback} and for all $(x_0,\xi_0) \in \mathcal{B}$.
Let
\begin{multline}
    \Psi(x,\xi, V_x, V_\xi) = \frac{1}{2}(q(x) + \Upsilon(x,\xi)) + \frac{1}{2} \tilde h(x,\xi)^\top \tilde h(x,\xi) \\
    + \begin{bmatrix} V_x^\top V_{\xi}^\top \end{bmatrix} \begin{bmatrix}
        f(x) + g(x) \tilde h(x, \xi) \\
        \alpha(x, \xi)
    \end{bmatrix},
    \label{eq:feedback_Hamiltonian}
\end{multline}
where $V: \mathbb{R}^n \times \mathbb{R}^{\nu} \rightarrow \mathbb{R}_{\geq 0}$ is a continuously differentiable positive-definite function.
Via the Dynamic Programming principle (\cite{Bellman1957}), \eqref{eq:augmented_cost} is minimised subject to the dynamics \eqref{eq:sysdyn}, \eqref{eq:dynamic-state-feedback-xi}, by $u^\star = {\arg \min}_{\tilde h(x,\xi)}  \Psi(x,\xi, V_x, V_\xi)$ satisfying 
\begin{equation}
    0 = \min_{\tilde h(x,\xi)} \Psi(x,\xi, V_x, V_\xi).
    \label{eq:HJB}
\end{equation}
The control law $u^\star$ corresponds to the ``classical" solution of an \emph{auxiliary, local} optimal control problem defined by the augmented cost \eqref{eq:augmented_cost} and the extended dynamics \eqref{eq:sysdyn}, \eqref{eq:dynamic-state-feedback-xi} in the region $\mathcal{B}$. It further constitutes (see, \emph{e.g.} \cite{Sassano2018,Sassano2021}) an approximate solution of the original optimal control problem (defined by \eqref{eq:sysdyn}, \eqref{eq:cost}), with two sources of approximation. 
Firstly, we are searching for a local solution in the region $\mathcal{B}$ of the (extended) state space $(x, \xi)$. Secondly, the solution is with respect to the modified (augmented) cost \eqref{eq:augmented_cost}, which features an additional running cost term $\Upsilon(x, \xi)$ depending on the system state $x$ as well as the dynamic extension $\xi$.

In this paper, we revisit the control problems defined above for the case in which the system dynamics \eqref{eq:sysdyn} are unknown, and propose a solution directly using measured input-state data.

\section{Data-driven system representation} \label{sec:data}
In the remainder of this paper, we consider \emph{unknown} input-affine systems. 
Let the following assumptions hold.
\begin{assumption}
    \label{as:unknown_but_well_behaved}
    The system dynamics are of the form \eqref{eq:sysdyn}. The mappings $f(x)$ and $g(x)$, while unknown, are such that Assumption~\ref{as:properties_dynamics} holds.
\end{assumption}
\begin{assumption}
    \label{as:basis functions}
    Libraries of basis functions $\mathcal{V}^f$ and $\mathcal{V}^g$, such that the unknown mappings $f(x)$ and $g(x)$ can be represented (exactly) as a linear combination of the elements of $\mathcal{V}^f$ and $\mathcal{V}^g$, respectively, are known.
\end{assumption}
Using Assumptions~\ref{as:unknown_but_well_behaved} and 
\ref{as:basis functions}, \eqref{eq:sysdyn} can be written as
\begin{align}
    \dot x & = F(x) x + g(x) u, \label{eq:sys_basis1}\\
    &= \tilde F Z(x) x + G \Xi(x) u, \label{eq:sys_basis2}
\end{align}
where $Z(x) \in \mathbb{R}^{\ell \times n}$ and $\Xi(x) \in \mathbb{R}^{p \times m}$ are matrices of known basis functions, and where the entries of $Z(x) x$ and $\Xi(x)$ are elements of $\mathcal{V}^f$ and $\mathcal{V}^g$, respectively.
Further, assume the following time-series data is available.
\begin{assumption}
    \label{as:data}
    Data of the state response $x_d$ and its rate of change $\dot x_d$ to a known exploring input $u_d$ can be collected at time instances $[ t_0, \ldots, t_{T-1} ]$.
\end{assumption}
Assumption~\ref{as:data} allows us to form the data matrices
\begin{subequations}
\begin{align}
    X_+ &= \begin{bmatrix} \dot x_d(t_0) & \ldots & \dot x_d(t_{T-1}) \end{bmatrix}, \\
    X_- &= \begin{bmatrix} x_d(t_0) & \ldots & x_d(t_{T-1}) \end{bmatrix}, \\
    U &= \begin{bmatrix} u_d(t_0) & \ldots & u_d(t_{T-1}) \end{bmatrix}.
\end{align}
Consider also the matrices
\label{eq:data_matrices}
\end{subequations}
\begin{subequations}
\begin{align}
   \mathcal{Z} &= \begin{bmatrix}  Z(x_d(t_0))x_d(t_0) & \ldots & Z(x_d(t_{T-1}))x_d(t_{T-1}) \end{bmatrix}, \\
    \mathcal{U} &= \begin{bmatrix} \Xi(x_d(t_0))u_d(t_0) & \ldots & \Xi(x_d(t_{T-1}))u_dt_{T-1}) \end{bmatrix},
\end{align}
which can be assembled from the data and the known matrices of basis functions $Z(x)$ and $\Xi(x)$. 
Note then that 
\label{eq:data_func_matrices}
\end{subequations}
\begin{equation}
    X_+ = \tilde F \mathcal{Z} + G \mathcal{U} = \begin{bmatrix}
        \tilde F & G
    \end{bmatrix} \begin{bmatrix}
        \mathcal{Z}\\
        \mathcal{U}
    \end{bmatrix}.
    \label{eq:data_matrix_relation}
\end{equation}

Consider the special case in which the feedback law \eqref{eq:dynamic-state-feedback} reduces to a static feedback law of the form
\begin{equation}
    u = h(x) = H(x) x.
    \label{eq:static-state-feedback}
\end{equation}
System \eqref{eq:sysdyn} in closed loop with \eqref{eq:static-state-feedback} evolves according to the dynamics
\begin{equation}
    \dot x =  f(x) + g(x)h(x) = \left( F(x) + g(x) H(x) \right)x.
    \label{eq:sysdyn_cl_static}
\end{equation}

\begin{lemma}
    \label{le:dd_sys_rep}
    Let Assumptions~\ref{as:unknown_but_well_behaved}, \ref{as:basis functions} and \ref{as:data} hold, and consider the matrices \eqref{eq:data_matrices}, \eqref{eq:data_func_matrices}. 
    If 
    \begin{equation}
    \text{rank} \left( \begin{bmatrix}
        \mathcal{Z} \\
        \mathcal{U}
    \end{bmatrix} \right) = \ell + p,
    \label{eq:rank_condition}
    \end{equation}
    then the ``closed-loop dynamics matrix'' $\left( F(x) + g(x) H(x) \right)$ can equivalently be represented using data as
    \begin{equation}
    \left( F(x) + g(x) H(x) \right) = X_+ \tilde G(x),
    \label{eq:dd_sysdyn_cl} 
    \end{equation}
    and the control matrix $H(x)$, can be represented as
    \begin{equation}
       H(x) = \Xi(x)^\dagger \mathcal{U} \tilde G(x),
        \label{eq:dd_H}
    \end{equation}
    where $\tilde G(x) \in \mathbb{R}^{T \times n}$ satisfies $Z(x) = \mathcal{Z} \tilde G(x)$.
\end{lemma}
\begin{proof}
    Note that
    \begin{align}
         \left( F(x) + g(x) H(x) \right) = \begin{bmatrix}
            \tilde F & G
        \end{bmatrix} \begin{bmatrix}
            Z(x) \\
            \Xi(x) H(x)
        \end{bmatrix}. 
        \label{eq:cl_reform}
    \end{align}
    If \eqref{eq:rank_condition} holds, then there exists $\tilde G(x)$ satisfying
    \begin{equation}
        \begin{bmatrix}
            Z(x) \\
            \Xi(x) H(x) 
        \end{bmatrix} = \begin{bmatrix}
            \mathcal{Z} \\
            \mathcal{U}
        \end{bmatrix} \tilde G(x).
        \label{eq:G_definition}
    \end{equation}
    Combining \eqref{eq:G_definition} with \eqref{eq:cl_reform} and using \eqref{eq:data_matrix_relation} gives \eqref{eq:dd_sysdyn_cl}. Finally, \eqref{eq:dd_H}  follows from the lower block row in \eqref{eq:G_definition}.
\end{proof}

In Section~\ref{sec:solution}, we utilise this data-driven representation of the closed-loop system under static state-feedback
to design \emph{dynamic} state-feedback strategies for the control problems introduced in Section~\ref{sec:problem}.

\begin{remark}
\label{re:system_representation}
The data-driven system representation in Lemma~\ref{le:dd_sys_rep} is inspired by the system representation for LTI systems introduced in \cite{DePersis2020}. If the input map $g(x)$ does not depend on the state $x$, i.e. $g(x) = G$ and $\Xi(x) = I_m$, the representation in Lemma~\ref{le:dd_sys_rep} 
is very similar to the representation introduced in \cite{Guo2020}, which like several related works (e.g. \cite{Gue2023,Liu2024,Liu2025,Dai2025}) focuses on input-affine systems with constant input map. 
General input maps are considered in \cite{DePersis2023} and \cite{Strässer2025} where the dynamics are lifted into a form with constant input map by adding an integrator to include the controller dynamics and using the Koopman operator, respectively, and in \cite{Guo2022,Madeira2024,Chen2025} in the context of robust sum of squares methods for polynomial systems, which admit an alternative system representation via a bound on unknown disturbances.
\end{remark}
\begin{remark}
\label{re:pseudoinverse}
The control matrix $H(x)$ in \eqref{eq:dd_H} depends on the pseudoinverse of the matrix of basis functions $\Xi(x)$. 
If $p<m$, $\Xi(x)^\dagger$ is a right inverse. In general, there may be infinite matrices $H(x)^\dagger$ such that the lower block row in \eqref{eq:G_definition} holds. If $\Xi(x)$ is square and invertible $\Xi(x)^\dagger = \Xi(x)^{-1}$. For general $p>m$, \eqref{eq:dd_H} is the least squares solution of the lower block row in \eqref{eq:G_definition}, and may hence not satisfy the condition exactly. Note, however, that $\Xi(x)$ is chosen by the designer. Hence, we can pick a suitable $\Xi(x)$ and $\Xi(x)^\dagger$ such that \eqref{eq:dd_H} is uniquely defined. The main constraint for the choice of dimension $p$ is the size of the library of basis functions $\mathcal{V}^g$ and hence the information we have about the function $g(x)$. 
In addition to the special case $\Xi(x) = I_m$ mentioned in Remark~\ref{re:system_representation}, $\Xi(x)$ can for example generally be chosen to be invertible if the functions in g(x) are known exactly up to a scaling factor.
\end{remark}

\section{Data-driven control via algebraic $P$ solutions} \label{sec:solution}
In this section, we combine the notion of algebraic $P$ solutions (\cite{Sassano2018}) with the data-driven system representation in Section~\ref{sec:data} to solve the problems defined in Section~\ref{sec:problem} for unknown input-affine systems. 

\subsection{Stabilisation} \label{subsec:stab}
Recall the stabilisation problem introduced in Section~\ref{subsec:probstab}. 
Finding a function $V: \mathbb{R}^n \times \mathbb{R}^{\nu} \rightarrow \mathbb{R}_{\geq 0}$ which satisfies the PDE \eqref{eq:stabilityPDE} is generally challenging. Inspired by \cite{Sassano2018}, consider the Lyapunov function candidate
\begin{equation}
    V(x, \xi) = \frac{1}{2} x^\top P(\xi) x + \frac{1}{2} (x-\xi)^\top R (x-\xi),
    \label{eq:Lyap}
\end{equation}
for $\xi \in \mathbb{R}^n$, where $R = R^\top \succ 0$ is a known constant matrix, which acts as a tuning parameter, and $P: \mathbb{R}^n \rightarrow \mathbb{R}^n\times\mathbb{R}^n$ is a matrix-valued mapping (to be designed), such that $P(\xi)=P(\xi)^\top \succ0$, for all $\xi$ in a neighbourhood containing the origin. Then, the first condition in \eqref{eq:lyap_las} holds and  
\begin{align}
    V_x &= P(\xi) x + R(x -\xi),\\
    V_{\xi} &= \frac{1}{2} \mathbf{J}_{x^\top P(\xi)}(\xi) x - R(x -\xi).
    \label{eq:partial_derivatives}
\end{align}
\begin{lemma}
    \label{le:stab_static_feedback}
    Let Assumptions~\ref{as:unknown_but_well_behaved}, \ref{as:basis functions} and \ref{as:data} hold and let $Y(x) \in \mathbb{R}^{T \times n}$, $\tilde P(x) =   \tilde P(x)^\top \succ 0 \in \mathbb{R}^{n \times n}$ be a solution of the feasibility problem
    \begin{align}
        X_+ Y(x) + Y(x)^\top X_+^\top &\prec 0, \label{eq:las_dd_inequality}\\
        Z(x) \tilde P(x) = \mathcal{Z} Y(x).
        \label{eq:las_dd_equality}
    \end{align}
    Then, for  sufficiently large $\kappa$,  the controller \eqref{eq:dynamic-state-feedback} with
    \begin{subequations}
    \label{eq:lad_dd_static_u}
    \begin{align}
        \tilde h(x,\xi) &= H(x)x = \Xi(x)^\dagger \mathcal{U}Y(x)P(x)x, \label{eq:dd_static_h}\\
        \alpha(x,\xi) &= - \kappa V_{\xi}, 
    \end{align}
    \end{subequations} 
    where $P(x) = \tilde P(x)^{-1}$, is such that the zero equilibrium of \eqref{eq:augmented_sys_dynamic_cl} is LAS.
\end{lemma}
\begin{proof}
    Since $\tilde P(\cdot)\succ 0$, the function $V(x,\xi)$ in \eqref{eq:Lyap}, with $P(\cdot)$ as given in the statement, is positive definite by construction. Then, the result in  \cite[Proposition 2.3]{Sassano2018} entails (via the main result in \cite{Anstreicher2000}) that there exists $\kappa^\star>0$ such that \eqref{eq:lyap_las} holds in a neighbourhood $\mathcal{N}$ containing the origin of \eqref{eq:augmented_sys_dynamic_cl} with $\tilde h = H(x)x$ and $\alpha$ as in \eqref{eq:lad_dd_static_u} for any $\kappa > \kappa^\star$, provided that the matrix-valued mapping $P(x)$ satisfies 
    \begin{equation}
        P(x)\left(F(x) + g(x) H(x) \right) + \left(F(x) + g(x) H(x) \right)^\top P(x) \prec 0.
        \label{eq:algPsolution_stab_matrix_eq}
    \end{equation}
    Via a congruence transformation with $P(x)^{-1}$, using
    Lemma~\ref{le:dd_sys_rep} and introducing the change of variable $Y(x) = \tilde G(x) \tilde P(x)$, \eqref{eq:algPsolution_stab_matrix_eq} gives \eqref{eq:las_dd_inequality}. Note that \eqref{eq:dd_sysdyn_cl} relies on condition \eqref{eq:G_definition}, which holds true if \eqref{eq:las_dd_equality} is satisfied and by the choice of $u$ as in \eqref{eq:lad_dd_static_u}.
\end{proof}
Lemma~\ref{le:stab_static_feedback} utilises the analysis result in \cite[Proposition 2.3]{Sassano2018} to design stabilising \emph{static} state-feedback controllers directly using data.
Next we show how the dynamic extension \eqref{eq:dynamic-state-feedback-xi} can be used more ``actively'' to design \emph{dynamic} 
feedback controllers from data.\\

\begin{lemma}
    \label{le:stab_dynamic_feedback}
    Let Assumptions~\ref{as:unknown_but_well_behaved}, \ref{as:basis functions} and \ref{as:data} hold and let $Y(x) \in \mathbb{R}^{T \times n}$, $\tilde P(x) =   \tilde P(x)^\top \succ 0 \in \mathbb{R}^{n \times n}$ be a solution of the feasibility problem \eqref{eq:las_dd_inequality}, \eqref{eq:las_dd_equality}. Then, for sufficiently large $\kappa$, the controller  
    \eqref{eq:dynamic-state-feedback} with
     \begin{subequations}
    \label{eq:lad_dd_dynamic_u}
    \begin{align}
        \tilde h(x,\xi) &= \Xi(x)^\dagger \mathcal{U}Y(x) \left(P(\xi)x + R(x-\xi)\right), \label{eq:dd_dynamic_h}\\
        \alpha(x,\xi) &= - \kappa V_{\xi},  \label{eq:dd_dynamic_xi}
    \end{align}
    \end{subequations}
    where $P(x) = \tilde P(x)^{-1}$, is such that the zero equilibrium of \eqref{eq:augmented_sys_dynamic_cl} is LAS. 
\end{lemma}
\begin{proof}
    Let $\Phi: \mathbb{R}^n \times \mathbb{R}^n \rightarrow \mathbb{R}^{n \times n}$ be a continuous mapping such that $P(\xi)x = P(x) x -\Phi(x,\xi)(x-\xi)$. Then \eqref{eq:dd_dynamic_h} can equivalently be written as 
    \begin{equation}
        \tilde h (x,\xi) = \Xi(x)^\dagger \mathcal{U}Y(x) \left(P(x) x + (R-\Phi(x,\xi))(x-\xi) \right).
        \label{eq:dd_dynamic_h_seperated}
    \end{equation}
    Note that the first term in \eqref{eq:dd_dynamic_h_seperated} is equal to $H(x)x$ in \eqref{eq:dd_static_h}. 
    Using this reformulation, recalling Lemma~\ref{le:dd_sys_rep} and following analogous arguments to \cite[Proposition 2.3 and Theorem 3.1]{Sassano2018} and 
    Lemma~\ref{le:stab_static_feedback}, it can be shown that if \eqref{eq:lad_dd_dynamic_u} is such that $P(x)$, $Y(x)$ satisfy \eqref{eq:las_dd_inequality}, \eqref{eq:las_dd_equality}
    there exists a $\kappa^\star>0$ such that \eqref{eq:lyap_las} holds in a neighbourhood $\mathcal{N}$ containing the origin of \eqref{eq:augmented_sys_dynamic_cl} for any $\kappa > \kappa^\star$. 
\end{proof}

\begin{remark}
    \label{re:why_is_this_interesting}
    Direct data-driven stabilisation of nonlinear systems has been intensively studied in recent years, see \cite{DePersis2023review} for an overview. Similar conditions to Lemma~\ref{le:stab_static_feedback} have for example been presented in \cite{Guo2020} using sum of squares arguments. The two alternative results presented in this section illustrate the approach which will allow us to design approximately optimal controllers directly using data in Section~\ref{subsec:optimal}. However, the data-driven dynamic state-feedback design in Lemma~\ref{le:stab_dynamic_feedback} is interesting in its own right, as it introduces the tunable parameters $\kappa$, $R$ and $\xi_0$, which can be designed to shape the behaviour of the resulting closed-loop system and the region of attraction of the zero equilibrium.
\end{remark}

\subsection{Approximate optimal control} \label{subsec:optimal}
Recall the approximate optimal control problem introduced in Section~\ref{subsec:approximate_optimal}.
Solving \eqref{eq:HJB} is generally challenging. Inspired by \cite{Sassano2018} and using Lemma~\ref{le:dd_sys_rep} we provide a solution via purely data-dependent algebraic conditions for the unknown system \eqref{eq:sysdyn}.
\begin{theorem}
    \label{th:optimal_control}
    Let Assumptions~\ref{as:cost}, \ref{as:unknown_but_well_behaved}, \ref{as:basis functions} and \ref{as:data} hold and let $Y(x) \in \mathbb{R}^{T \times n}$, $\tilde P(x) =   \tilde P(x)^\top \succ 0 \in \mathbb{R}^{n \times n}$ be a solution of the feasibility problem
    \begin{align}
    \label{eq:oc_dd_inequality}
        \begin{bmatrix}
            X_+ Y(x) + Y(x)^\top X_+^\top & (\Xi(x)^\dagger \mathcal{U}Y(x))^\top  & \tilde P(x) C(x)^\top \\
            \Xi(x)^\dagger \mathcal{U}Y(x) & -I_m & 0 \\
            C(x)\tilde P(x) & 0 & -I_\mu
        \end{bmatrix} \prec 0,
    \end{align}
    and \eqref{eq:las_dd_equality}. Then, there exists a constant $\kappa^\star$ and an open set $\mathcal{B}$, such that for all  $\kappa > \kappa^\star$ and all $(x, \xi) \in \mathcal{B}$, \eqref{eq:HJB} holds with
    $\tilde h(x,\xi)$ and $\alpha(x,\xi)$ as defined in \eqref{eq:lad_dd_dynamic_u}, where $P(x) = \tilde P(x)^{-1}$, and with the function $V$ given by \eqref{eq:Lyap}, for some 
    $\Upsilon(x,\xi) > 0$. 
    Hence, the the controller given by 
    \eqref{eq:dynamic-state-feedback}, \eqref{eq:lad_dd_dynamic_u} solves the 
    auxiliary, local optimal control problem defined in Section~\ref{subsec:approximate_optimal}.
\end{theorem}
\begin{proof}
    By Lemma~\ref{le:dd_sys_rep}, the change of variable $Y(x) = \tilde G(x) \tilde P(x)$, and a congruence transformation with $P(x)^{-1}$, \eqref{eq:oc_dd_inequality}, \eqref{eq:las_dd_equality} are such that
    \begin{multline}
        C(x)^\top C(x) + H(x)^\top H(x) + P(x) \left( F(x) + g(x) H(x)\right) \\ + \left( F(x) + g(x) H(x)\right)^\top P(x) \prec 0,
        \label{eq:oc_matrix_ineq}
    \end{multline}
    for $H(x)$ as defined in \eqref{eq:dd_static_h}. Using \eqref{eq:dd_dynamic_h_seperated} and arguments analogous to \cite[Theorem 3.1]{Sassano2018}, there exists $\kappa^\star$ such that 
    \begin{multline}
        - \frac{1}{2}\Upsilon(x,\xi) := \frac{1}{2}q(x)  + \frac{1}{2} \tilde h(x,\xi)^\top \tilde h(x,\xi) \\ + V_x^\top (f(x) + g(x) \tilde h(x, \xi)) - \kappa V_{\xi}^\top V_{\xi} < 0,
        \label{eq:Uppsilon}
    \end{multline}
    for all $\kappa > \kappa^\star$ and $(x,\xi) \in \mathcal{B}$, via the main result in \cite{Anstreicher2000}. 
    Finally, consider $V$ as in \eqref{eq:Lyap} as a candidate Lyapunov function. Then, 
    $$ \dot V = V_x^\top (f(x) + g(x) \tilde h(x, \xi)) - \kappa V_{\xi}^\top V_{\xi} \leq -\frac{1}{2}\Upsilon(x,\xi) < 0,$$
    by construction. Hence, the zero equilibrium of system \eqref{eq:augmented_sys_dynamic_cl} with $\tilde h$, $\alpha$ as in \eqref{eq:lad_dd_dynamic_u}, is LAS. 
\end{proof}

\begin{remark}
    \label{re:nonstrict_inequalities}
    The strict inequalities \eqref{eq:las_dd_inequality} and \eqref{eq:oc_dd_inequality} could be replaced with nonstrict inequalities subject to the additional constraint that the matrices on the left hand side of \eqref{eq:las_dd_inequality} and \eqref{eq:oc_dd_inequality} evaluated at $x=0$ are positive definite, which is in line with the definition of algebraic $P$ solutions in \cite{Sassano2018}, and may lead to potentially less conservative solutions with larger region of attraction of the zero equilibrium of the closed-loop system. Proving local asymptotic stability then involves LaSalle's invariance
    principle and zero-state detectability, see \cite{Sassano2021} and references therein for more details. For ease of exposition, we instead consider strict inequalities.
\end{remark}

\begin{remark}
    \label{re:minimise_degree_of_approximation}
    The result of Theorem~\ref{th:optimal_control} provides a method to design data-driven dynamic state-feedback controllers which are exact solutions of the auxiliary, local problem defined in Section \ref{subsec:approximate_optimal} and hence approximate solutions of the ``classical'' optimal control problem defined by \eqref{eq:sysdyn}, \eqref{eq:cost},
    with guaranteed stability and performance (locally). The additional cost arising from the presence of the dynamic extension \eqref{eq:dd_dynamic_xi} can be minimised via a suitable initialisation, \emph{i.e.} choosing $\xi_0 \in \text{\emph{arg}}\min_\xi V(x_0,\xi)$, for a given initial condition $x_0$ of system \eqref{eq:sysdyn}. Additionally, by allowing the gain $\kappa$ to vary, 
    \emph{i.e.} allowing the dynamic extension 
    to evolve according to dynamics of the form $\dot \xi = - \kappa(x, \xi)V_{\xi}$, 
    the additional running cost term $\Upsilon(x,\xi)$ can be minimised. More precisely, if $\kappa(x,\xi)$ is chosen as in \cite[Theorem 3.2]{Sassano2018}, $\Upsilon(x,\xi) \rightarrow 0$ apart from on a certain sub-manifold which can be rendered arbitrarily small and in which $\Upsilon(x,\xi)$ is bounded. Note that the model-based conditions in \cite[Theorem 3.2]{Sassano2018} can be parametrised using data using Lemma~\ref{le:dd_sys_rep}.
\end{remark}

\begin{remark}
    \label{re:noise}
    While we consider the noise-free case, \emph{i.e.} system \eqref{eq:sysdyn} and data \eqref{eq:data_matrices} are not affected by noise, the underlying approach extends to other control problems involving the solution of PDEs, including
    $\mathcal{L}_2$-disturbance attenuation (\cite{Sassano2012}). 
    An important question in data-driven control is how to represent unknown systems based on noisy data. Considering process noise, several system representation  approaches have been proposed, see e.g. \cite{DePersis2023review}. We expect that these can be applied to derive alternative data-driven algebraic conditions for the problems considered herein.
    Considering measurement noise, if the type of system dynamics is known, the approach to treat measurement noise similarly to process noise introduced in \cite{DePersis2020} extends without difficulties to nonlinear systems (\cite{DePersis2023review}). 
    Like many related results (\emph{e.g.} \cite{Guo2022,Madeira2024}), 
    our results depend on measurements of 
    $\dot x$, which are often unavailable in practice and might have to be approximated via filters. Hence, the samples $\dot x_d$ are most likely to be affected by measurement noise. While a system representation as in \cite{Guo2022} can be used instead, note that the presented methods have an inherent degree of robustness as demonstrated in the example in Section~\ref{sec:example}. This is due to the construction based on the Lyapunov function \eqref{eq:Lyap}.
\end{remark}

\section{Example} \label{sec:example}
Consider the controlled van der Pol oscillator, \emph{i.e.} \eqref{eq:sysdyn} with
\begin{align*}
    f(x) = \begin{bmatrix}
         x_2 \\
        -x_1 + (1-x_1^2)x_2
    \end{bmatrix}, \ \ \
    g(x) = \begin{bmatrix}
        0 \\
        1
    \end{bmatrix},
\end{align*}
and the cost \eqref{eq:cost} with $q(x) = 1.25x_1^2 + 2x_2^2+6x_1^2x_2^2$.
Note that $f$, $g$, and $q$ satisfy Assumptions~\ref{as:properties_dynamics} and \ref{as:cost} and that the equilibrium at the origin of the open-loop system $\dot x = f(x)$ is unstable. 
For this simple example, we can determine a closed-form solution to the ``classical'' optimal control problem \eqref{eq:sysdyn}, \eqref{eq:cost}, (using model knowledge) by solving the corresponding HJB PDE (see \emph{e.g.} \cite{liberzon2011}). More precisely, the static feedback law $u^* =-0.5x_1 - 3x_2$ is the optimal control input with value function $V(x) = 2x_1^2 + 0.5x_1x_2 + 1.5x_2^2+0.125x_1^4$.
Let $f$ and $g$ be unknown and let Assumption~\ref{as:data} hold. Consider the libraries of basis functions $\mathcal{V}^f = \mathcal{V}^g = \mathcal{V}$, with $\mathcal{V}$ containing monomials up to order two. The system can be written in the form \eqref{eq:sys_basis2}. We choose
\begin{align*}
    Z(x) = \begin{bmatrix}
        x_1x_2  & 0 \\
    1   & -1 \\
    0 & 1 \\
    \end{bmatrix}, \ \
    \Xi(x) = \begin{bmatrix}
        1 \\
    \end{bmatrix}.
\end{align*}
We excite the system by choosing $u_d = 10\cos(20t)$
starting from $x_0=\begin{bmatrix}
    0.1 & 0.5
\end{bmatrix}^\top$ and collect $T = \ell + p = 4$ data samples $x_d$ and $\dot x_d$ with a sampling time of $\Delta = 0.01$s to form the matrices \eqref{eq:data_matrices}, \eqref{eq:data_func_matrices}. A solution to the data-driven feasibility problem \eqref{eq:oc_dd_inequality}, \eqref{eq:las_dd_equality} is computed using SOSTOOLS (\cite{sostools}). The tunable parameters are chosen as $R = I_n$, $\kappa = 2$ and $\xi_0 = \begin{bmatrix}
    0.1 & 0.1
\end{bmatrix}^\top$. 
Figure~\ref{fig:xandu} shows the time histories of the control input, the dynamic extension and the state of the resulting closed-loop system starting from $x_0 = \begin{bmatrix}
    0.1 & 0.1
\end{bmatrix}^\top$ for three different controllers. That is, $u_{dd}(x,\xi)$ obtained using the result of Theorem~\ref{th:optimal_control}, $u_{dd}^{\text{noisy}}(x,\xi)$ obtained using the result of Theorem~\ref{th:optimal_control} with the entries of $X_+$ computed via a finite-difference approximation, i.e. $\dot x_d(t_k) = (x_d(t_{k+1})-x_d(t_k))/\Delta$ for $k = 0, \ldots, T-1$, and $u^*(x)$ defined above. The corresponding costs (rounded to two decimal places) are $J(x_0,u^*) = 0.04 $, $J(x_0,\xi_0,u_{dd}) = 0.06$, and $J(x_0,\xi_0,u_{dd}^{\text{noisy}}) = 0.27$. 
Despite not explicitly accounting for noisy derivative measurements, the controller $u_{dd}^{\text{noisy}}(x,\xi)$ designed using Theorem~\ref{th:optimal_control} with noise-corrupted $X_+$ stabilises the zero equilibrium of the resulting closed-loop system. However, it results in a higher cost. Using noise-free measurements, the controller $u_{dd}(x,\xi)$ designed using Theorem~\ref{th:optimal_control} stabilises the zero equilibrium of the resulting closed-loop system and results in an only slightly higher cost than the known solution to the ``classical problem'' $u^\star(x)$. 
\begin{figure}[t!]
    \centering
    \vspace{-0.75cm}
    \includegraphics[width=0.99\linewidth]{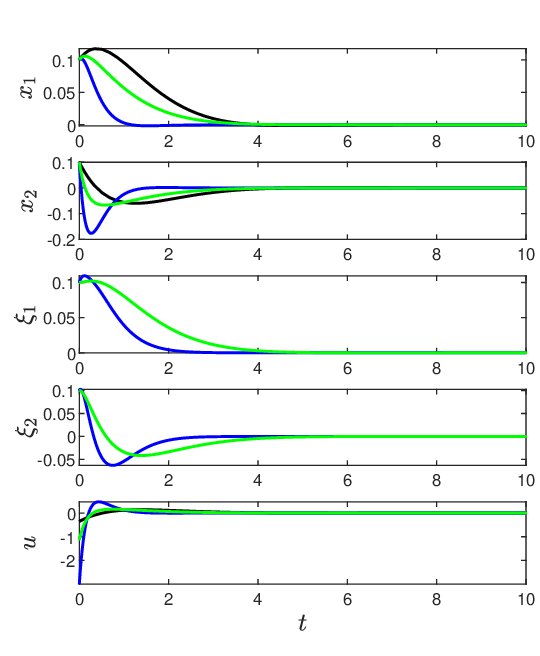}
     \vspace{-1cm}
    \caption{Time-histories of the state, dynamic extension and input for the Van der Pol example with $u = u_{dd}(x,\xi)$ (green), $u = u_{dd}^{\text{noisy}}(x,\xi)$ (blue), 
    $u = u^*(x)$ (black).}
    \vspace{0.4cm}
    \label{fig:xandu}
\end{figure}

\section{Conclusion} \label{sec:conclusion}
We consider the problem of designing approximately optimal controllers for optimal control problems involving nonlinear, input-affine systems with unknown dynamics directly using data. First, we introduce a direct data-driven system representation that is amenable to our problem formulation. This is then utilised in the context of a framework to systematically construct locally stabilising controllers and approximate solutions to optimal control problems via system immersion and algebraic conditions. 
The presented methods do not require the explicit solution of any nonlinear PDE, but solutions are instead characterised via purely data-dependent algebraic conditions. 
We comment on the sources of approximation and the effect of noisy measurements on the solution. The efficacy of the results is illustrated via a numerical example.


\balance
\bibliography{refs}                                                   
\end{document}